\documentclass[11pt]{article}

\usepackage{amsthm}
\usepackage{amsmath,amsfonts,amssymb,latexsym}

\newtheorem{thm}{Theorem}

\newtheorem{lem}[thm]{Lemma}

\newtheorem{prop}[thm]{Proposition}

\DeclareMathOperator{\rank}{rk}

\DeclareMathOperator{\Hom}{Hom}

\title{An addendum to `A rigidity property for the set of all characters induced by valuations'}
\author{J.~R.~J.~Groves}

\begin{document}
\maketitle
\section{Introduction}
The aim of this note is to correct and amplify some comments made in \cite{rig} on a possible generalisation to the results contained there. The comments are contained in the second paragraph of page 427. The paper is concerned with field automorphisms preserving finitely generated multiplicative subgroups of the field. The first sentence of the offending paragraph claims that the techniques can be extended to  subgroups which are torsion-free and of finite rank. This is largely correct and in the following section, I make precise the result to which this leads (Theorem \ref{thmb}). The second sentence claims that a major result of the paper holds in this wider context. This is false and  I will describe some results which might replace the false assertion (Theorem \ref{cor}).

 Segal in \cite[Proof of Theorem 2]{seg2} first observed  the incorrectness of our claim. This note depends heavily on his observation and on  some of the results he proved while considering related problems. Brookes, who made a similar error in \cite{cjbb}, has also proved a replacement to the incorrect claim. This is reported in \cite{seg2} in Section 3.1. Thus most of this note does not describe new results. In Section  \ref{extrathm}, however, we attempt to narrow down  further the new possibilities that were overlooked in the original paper (Theorem \ref{xthm}).
 
\section{Statement of results}\label{SR}
We introduce some notation. (In general we shall prefer to retain the notation and context of \cite{rig} rather than those of \cite{seg1}, \cite{seg2} or \cite{cjbb}.)  We denote by $K$  a field with subfield $k$. If $G$ is a torsion-free group of finite rank which is a subgroup of the multiplicative group of $K$ then we denote by $kG$ the group ring of $G$ over $k$ and by $k[G]$ the image of this group ring in $K$---alternatively, the $k$-subalgebra of $K$ generated by $G$. We shall suppose that any element of $G$ which is algebraic over $k$ is trivial. This assumption  is probably not necessary (it is not needed to prove the first theorem when $G$ is finitely generated) but the extra generality would take us too far from our main path.

A subgroup $H$ of a group $G$ is said to be {\it dense}  in $G$ if $G/H$ is torsion.

The {\it controller} $\mathcal C_G(P)$ of an ideal $P$ of a group ring $kG$ is the unique minimal subgroup $H$ such that  $P=(kH\cap P)kG$. We shall denote by $\mathcal B_G(P)$ the isolator of the controller of $P$ (that is, the subgroup of elements of $G$ which have some power lying in the controller of $P$). 

We now fix a subgroup $G$ of the multiplicative group of $K$ which has finite rank. Fix $I$ to be the kernel of the projection from $kG$ onto $k[G]$ and let $B=\mathcal B_G(I)$. If $G$ is finitely generated then $B$ will coincide with the subgroup with the same name in \cite{rig} (see the discussion in Section \ref{PfThmB}). In the more general case of finite rank, a natural extension of the definition in \cite{rig} yields the same result as given here. (See Proposition \ref{techprop}.)

If $\sigma$ is a $k$-automorphism of $k[G]$ which stabilises $G$ then the restriction of $\sigma$ to $G$ can be extended to an automorphism of $kG$ which stabilises $I$. It is then clear that $\sigma$ must also stabilise the subgroups $\mathcal C_I(G)$ and $\mathcal B_I(G)$ of $G$. Observe that, because $G$ is torsion-free of finite rank, its Malcev completion ($G\otimes \mathbb Q$ in additive notation) is a finite dimensional vector space over the rational numbers. Thus it is reasonable to use the language of vector spaces in this context.

\begin{thm}[cf. Theorem B of \cite{rig}]  \label{thmb}
The subgroup $B$ has a finite set $\mathcal Y$ of rank one subgroups which generate a dense subgroup of $B$ and so that any group of $k$-automorphisms of $k(G)$ which stabilises $G$ will have a subgroup of finite index which stabilises each element of $\mathcal Y$.

\end{thm}

Note that, in particular, any group of $k$-automorphisms which stabilises $G$ will have a subgroup of finite index which acts diagonally on $B$ with respect to a basis consisting of elements from the subgroups contained in $Y$. The subgroup induced within the automorphism group of $B$ will therefore be virtually abelian.

We prove Theorem \ref{thmb} (modulo the technical Proposition \ref{techprop}) in Section \ref{PfThmB}. The proof is a relatively straightforward extension of the proof of Theorem B of \cite{rig}. 
\begin{thm}[cf. Corollary of  \cite{rig}] \label{cor} Suppose that $\Gamma$ is a finitely generated group of $k$-automorphisms of $k(G)$ and that $\Gamma_B$ is the group of automorphisms of $B$ induced by $\Gamma$. Let $\Gamma_1$ denote the subgroup of finite index of $\Gamma_B$ consisting of those elements which have diagonal action on $B$.
\begin{itemize}
\item If $k$ has characteristic zero then $\Gamma_1$ is finite.
\item If $k$ has characteristic $p$ then there is a maximal independent subset $\mathcal V$ of $B$ so that,  if $\gamma\in \Gamma_1$ and $b\in \mathcal V$ then there exists an integer $n$ so that $\gamma(b) =b^{p^n}$.
\end{itemize}
\end{thm}

This result has also been proved by Brookes  and, in a slightly different form,  by Segal\cite[Section 3.1]{seg2}.

It is clear that $p$th powering is always a monomorphism (at least) of a field of characteristic $p$. How much can the automorphisms of Theorem \ref{cor} differ from this? 

Suppose that $G$ is the direct sum of 4 copies of the $p$-adic rationals generated (as module over the $p$-adic rationals) by $\{x,y,u,v\}$. Let $I$ be the prime ideal of $kG$ generated by the closure, under all powers of the $p$th powering automorphism, of $x+y+1$ and $u+v+1$. It is clear that $I$ is invariant under both the map which sends $x$ and $y$ to their $p$th powers and fixes $u,v$ and the map which fixes $x,y$ and sends  $u$ and $v$ to their $p$th powers. Thus, in this case, the group $\Gamma$ is not virtually cylic.

This occurs, however, because $kG/I$ has a tensor product decomposition and our aim is to show that something of this kind is always the case.

\begin{thm} \label{xthm} Suppose that $k$ is the field with $p$ elements. Suppose that $G$ is an extension of a finitely generated group by a $p$-group. Let $\gamma$ be an automorphism of $k(G)$ which stabilises $G$ and  which has restriction to $B$ denoted by $\gamma_1$. Then either some power of $\gamma_1$ coincides with $p^n$th powering for some integer $n$ or $B$ has a dense subgroup $B_1$ so that $B_1$ is a  non-trivial direct product $B_1\cong C \times D$ and $k[B_1]$ is a free join of $k[C]$ and $k[D]$. 
\end{thm}

The {\it free join} is described in Zariski and Samuel [III,15, Definition 1]\cite{zs}. In this context, it means that a transcendence subset (over $k$) of $k[C]$ joined with a transcendence subset of $k[D]$ is still a transcendence subset. The free join is very close to the tensor product in that a free join of two integral domains is always the quotient of the corresponding tensor product by a prime ideal which consists entirely of zero-divisors.

It is fairly clear that when $k[B]$ is a free join as described then  it is possible to produce automorphisms which are not simple $p^n$th powering for some fixed $n$.

\section{Proof of Theorem \ref{thmb}}\label{PfThmB}

The following proposition is very similar to Proposition 2.1 of  \cite{rig}. 

\begin{prop}\label{techprop} Let $H$ be a subgroup of $G$ so that $G/H$ is torsion-free. Then the following are equivalent:

\begin{enumerate}
\item[\textnormal{(i)}] $I$ is controlled by $H$;
\item[\textnormal{(ii)}] $k[G]$ is induced from $k[H]$;
\item[\textnormal{(iii)}] The transcendence degree of $k[G]$ over $k[H]$ is equal to the rank of $G/H$;
\item[\textnormal{(iv)}]The complete inverse image, under the restriction map, of $\Delta(H)$ is $\Delta(G)$.
\end{enumerate}
\end{prop}

We shall defer the proof of this to Section \ref{PfTechProp}. It is technical and, given the proof of the corresponding result in  \cite{rig}, not very surprising.

Observe that it follows quickly from Proposition \ref{techprop} that $B$, the isolator of the controller of $I$, is also the unique least isolated subgroup of $G$ satisfying (iii) of the proposition. Thus  the definition we have given here of the subgroup $B$ coincides with that of \cite{rig}.

\begin{proof}[Proof of Theorem \ref{thmb}] We shall denote the dual $\Hom(E, \mathbb R)$ of a group $E$ by $E^*$. Let $G_0$ be a finitely generated dense subgroup of $G$. Then the embedding of $G_0$ into $G$ induces an isomorphism of $G^*$ with $G_0^*$. Also, as a valuation on $k[G_0]$ will extend  to a valuation on $k[G]$, this isomorphism will identify $\Delta(G)$ with $\Delta(G_0)$.

 Observe that $B_0=B\cap G_0$ is isolated in $G_0$ and the transcendence degree of $k[G_0]$ over $k[B_0]$ is equal to the rank of $G_0/B_0$ since both are equal to the rank of $G/B$. Thus $B_0$ is the isolator of the controller of $I\cap kG_0$ in $G_0$. We wish to apply Theorem B of \cite{rig}. This theorem assumes that there no algebraic elements in $G_0$, which is clear from our assumptions, and that $B_0=G_0$. (Note that there is another inconsistency in \cite{rig} in that the definitions imply that $G/B$ is torsion-free whereas the theorem requires $G/B$ to be finite rather than trivial. But this causes no major problems.) If $B_0<G_0$ then we simply work with $B_0$ replacing $G_0$ and deduce from Theorem B of \cite{rig} that  there exists a finite set $\mathcal X$ of one dimensional subspaces of $B_0^*$ which are one-dimensional intersections of spaces  in the Bergman carrier of $\Delta(B_0)$. But  $B_0^*=B^*$ and $\Delta(B_0)=\Delta(B)$ so that this set of subspaces is determined also by the Bergman carrier of $\Delta(B)$. 

Any  group of $k$-automorphisms of $k(G)$ which stabilise $G$ must also stabilise $B$ and hence $\Delta(B)=\Delta(B_0)$. Thus it will also stabilise the (finite) Bergman carrier of $\Delta(B)$ and, finally, also the set $\mathcal X$ of one dimensional intersections arising from this Bergman carrier. Thus a subgroup of finite index will stabilise each element of $\mathcal X$.

Finally, observe that we can construct from $\mathcal X$ a finite set of subspaces of $B^*$ of co-dimension one which intersect trivially. Then,  by looking at the duals of these subspaces, we have a finite set $\mathcal Y$ of subgroups of $B$ having rank one and which generate a dense subgroup of $B$. The fact that each element of $\mathcal X$ is stabilised will imply also that each element of $\mathcal Y$ is stabilised,  as required.

\end{proof}

\section{ Proof of Theorem \ref{cor}}\label{PfCorB}

The elements of $\Gamma_1$ have diagonal action on $B$ and so, because $\Gamma_1$ is finitely generated, the eigenvalues of elements of $\Gamma_1$ can involve only finitely many primes. 
 
Observe that the controller of $I$ is dense in $B$  and let $H$ be some finitely generated dense subgroup of the controller of $I$. Hence $H$ is also dense in $B$. Let $H_1$ denote  the $\Gamma$-closure of $H$ in $B$. Then $H_1$ is a minimax group and so we can apply the theorems of Segal in \cite{seg1},\cite{seg2}. 

Suppose that $k$ has characteristic zero. Then Theorem 1.1 of Segal\cite{seg1} tells us that the controller of $I\cap kH_1$ is finitely generated. There is no loss in  assuming that this controller is $H$ itself. But  the controller is clearly stabilised by any automorphism of $B$ which stabilises $I$. Thus $\Gamma_1$ induces a group of diagonal automorphisms with integer coefficients on the finitely generated group $H$ and so must act finitely on $H$. Thus $\Gamma_1$ also acts finitely on the isolator, $B$, of $H$, as required.

Suppose that $k$ has finite characteristic $p$. We follow the argument of Segal in the proof of Theorem 2 of \cite{seg2}. Using Lemma 6, we have that $H_1$ has a subgroup $C$ which is an extension of a finitely generated group by a $p$-group so that $C$ controls $I$. (Observe that the proof of Lemma 6 does not require that $k$ be finite). Thus the controller of $I$ in $H_1$ is also an extension of a finitely generated group by a $p$-group and is invariant under the action of $\Gamma_1$. Let $\Gamma_2$ be the subgroup of finite  index in $\Gamma_1$ which consists of diagonal automorphisms. Since $H_1$ is an extension of a finitely generated group by a $p$-group, the eigenvalues of each of these automorphisms must be (plus or minus) powers of $p$. The second claim of  Theorem \ref{cor} follows.

 \section{Proof of Theorem \ref{xthm}}\label{extrathm}
 
We shall assume, as in the statement of Theorem  \ref{xthm}, that $k$ is the field with $p$ elements (so that $p$th powering is a $k$-automorphism) and that $G$ is an extension of a finitely generated subgroup by a $p$-group.

\begin{lem}\label{1lemma} Let $k\subseteq L=k(X)\subseteq K$  and suppose that $X$ is a transcendence  set over $k$. Let $\gamma$ be a $k$-monomorphism of $L$ so that, for each $x\in X$, $\gamma(x)=x^{p^{n_x}}$ for some non-negative integer $n_x$. 

Let $N$ be an integer satisfying $N>n_x$ for all $x\in X$. If, for some $a\in L$, we have $\gamma(a)=a^{p^N}$ then $a\in k$.
\end{lem}

\begin{proof} Suppose that $a=u/v$ with $u,v\in k[X]$. Observe that, as $X$ is a transcendence set, $k[X]$ is a polynomial ring over the field $k$ and so is a unique factorisation domain. Thus we can assume that $u$ and $v$ have no common factor. If $\gamma(a)=a^{p^N}$ then 
$$\gamma(u)v^{p^N}=\gamma(v)u^{p^N}$$
Thus $v^{p^N}$ divides $\gamma(v)u^{p^N}$ but has no factor in common with $u^{p^N}$. Hence $v^{p^N}$ divides $\gamma(v)$. Suppose that the total degree of $v$ is $d$. Then the total degree of $v^{p^N}$ is $dp^N$.  But the total degree of $\gamma(v)$ is bounded by the maximum of  $vp^{n_x}$ for  $x\in X$ and this maximum is less than $dp^N$ unless $d=0$. That is, $v\in k$. A similar argument shows that $u\in k$ and so $a\in k$, as required.
\end{proof}

\begin{lem} \label{2lemma} Let $k \subseteq L\subseteq K$  and let $\gamma$ be a $k$-monomorphism of $K$ so that $\gamma(L)\subseteq L$.  Suppose that $X$ is a subset of $K$ so that $\gamma$ acts on $X$ by $p^n$th powering. Suppose also that, if $l\in L$ and $\gamma(l)=l^{p^n}$ then $l \in k$. 

If $X$ is a transcendence set over $k$ then it is also a transcendence set over $L$.
\end{lem}

\begin{proof} Let $\mathcal M$ denote the set of monomials in $X$. Suppose that $X$ is not a transcendence set over $L$ and let 
$$\sum_{m\in \mathcal M} l_m m=0$$ be a non-trivial expression, with $l_m\in L$, which expresses this fact. Suppose also that amongst all such expressions, our chosen one  involves the least number of monomials. Choose $m_0\in \mathcal M$ so that $l_{m_0}\neq 0$.

Applying first $\gamma$ and then $p^n$th powering to the expression, we obtain, respectively,
$$ \sum_{m\in \mathcal M} \gamma(l_m) m^{p^n} =0\quad\text{ and }\quad \sum_{m\in \mathcal M} l_m^{p^n} m^{p^n} =0$$
Eliminating $m_0^{p^n}$ between these two expressions, we obtain
$$ \sum_{m\in \mathcal M}( \gamma(l_m)l_{m_0}^{p^n}-\gamma(l_{m_0})l_m^{p^n})m^{p^n} =0$$
This is an expression which involves fewer monomials than the chosen one and so must be trivial. That is, for all $m$,

$$ \gamma(l_m)l_{m_0}^{p^n}-\gamma(l_{m_0})l_m^{p^n}=0$$
and so

$$ \gamma\left(\frac{l_m}{l_{m_0}}\right)=\left(\frac{l_m}{l_{m_0}}\right)^{p^n}$$
Thus, by our assumption,

$$\frac{l_m}{l_{m_0}}=k_m\in k$$
and so our original expression can be rewritten as
$$\sum_{m\in \mathcal M} k_ml_{m_0} m=l_{m_0}\left(\sum_{m\in \mathcal M} k_m m\right)=0$$
It follows that $\sum_{m\in \mathcal M} k_m m=0$ and so $X$ is not a transcendence set over $k$, contrary to hypothesis. The proof is complete.
\end{proof}

 \begin{prop} \label{xprop} Let $\gamma$ be a $k$-monomorphism of $K$ and let $X_1, \dots, X_e$ be subsets of $K$ so that the action of $\gamma$ on $X_i$ is to replace each element by its $p^{n_i}$th power. Suppose that the $n_i$ are distinct non-negative integers. If the sets $X_i$ are transcendental over $k$ then so also is $X=\cup_i X_i$.
\end{prop}
\begin{proof} Arrange the $X_i$ so that $n_1<n_2<\dots<n_e$. By an evident induction, we can assume that 
$$Y=\bigcup_{i=1\dots e-1} X_i$$
is a transcendence set over $k$. Let $L=k(Y)$. Then  $L\supseteq \gamma(L)$ and so $L$ and $\gamma$ satisfy the hypotheses of  Lemma \ref{1lemma}. Thus, if we take $N=n_e$ and $a\in L$, then we can deduce that, if $\gamma(a)=a^{p^{n_e}}$ then $a\in k$. 

Thus the hypotheses of Lemma \ref{2lemma} are also satisfied and, as $X_e$ is a transcendence set over  $k$, by choice, it is also a transcendence set over $L$. That is, $Y\cup X_e=\cup_{i=1\dots e} X_i$ is  a transcendence set over $k$.
\end{proof}

\begin{proof}[Proof of Theorem \ref{xthm}]  The theorem is almost contained in Proposition \ref{xprop}. We need first to make some adjustments, however. 
By Theorem \ref{cor}, some power $\gamma_2=\gamma_1^n$ of $\gamma_1$ acts diagonally on $B$ with eigenvalues which are powers of $p$.  

If $\gamma_2$ has eigenvalues which are negative powers of $p$ then combine $\gamma_2$ with $p^m$th powering, for some sufficiently large $m$, to obtain a monomorphism $\gamma_3$ for which all of the eigenvalues are non-negative powers of $p$. Notice that the number of distinct eigenvalues remains unchanged after this combination.

Let $E_1, \dots, E_k$ be the intersection, with $B$, of the eigenspaces of $\gamma_3$. Observe that the $E_i$ generate their direct sum $B_1$ and $B_1$ is dense in $B$. For each $i$, let $X_i$ be a transcendence set in $k[E_i]$. By Proposition \ref{xprop}, $X_1\cup \dots \cup X_k$ will also be a transcendence set over $k$. But then $k[B_1]$ is a free join of the subalgebras $k[E_i]$.

That is, either $k[B_1]$ is a non-trivial free join or there is only one eigenspace for $\gamma_2$. In the latter case, $\gamma_2$ must be $p^n$th powering for some $n$.
\end{proof}

\section{Appendix: Proof of Proposition \ref{techprop}} \label{PfTechProp}

The following observation is easily verified for a subgroup $H$ of $G$.
\begin{quote} 
There exists a finitely generated subgroup $X$ of $G$ so that $H\cap X=\{1\}$ and $HX$ is dense in $G$.
Further, if $G/H$ is torsion-free and if $G_1/H$ is finitely generated, then we can choose $X$ so that $HX=G_1$.
\end{quote}

\begin{proof}

\noindent $\textnormal{(i)} \iff\textnormal{(ii)}$

The kernel of the map $kH \rightarrow k[H]$ is clearly $I\cap kH$ and so there is a short exact sequence
$$ I \cap kH  \rightarrowtail kH \twoheadrightarrow k[H]$$
Since $kG$ is free when considered as $kH$-module, we can tensor this exact sequence over $kH$ with $kG$ to obtain an exact sequence
$$ (I \cap kH)\otimes_{kH} kG  \rightarrowtail kH\otimes_{kH} KG \twoheadrightarrow k[H]\otimes_{kH} KG$$that is,
$$ (I \cap kH)kG  \rightarrowtail KG \twoheadrightarrow k[H]\otimes_{kH} KG$$
Thus, $I=(I\cap kH)kG$ if and only if $k[G]=k[H]\otimes_{kH} KG$; that is, $I$ is controlled by $H$ if and only if $k[G]$ is induced from $k[H]$.

\noindent $\textnormal{(iii)} \Rightarrow \textnormal{(ii)}$

Observe that $k[G]$ is induced from $k[H]$ if and only if , for every subgroup $G_1$ of $G$ with $G_1/H$ finitely generated, $k[G_1]$ is induced from $k[H]$. Thus it will suffice to show that, if $G_1/H$ is finitely generated and $G_1$ is dense in $G$ then $k[G_1]$ is induced from $k[H]$. By the remark at the start of the Section, we can assume that $G_1=HX$ with $H\cap X=\{1\}$. 

Since $G_1$ is dense, $k[G_1]$ has the same transcendence degree as $k[G]$. Also the rank of $X$ must be the rank of $G/H$. Thus we have that the transcendence degree of $k(H)[X]$ over $k(H)$ equals the rank of $X$. But then $k(H)[X]$ is the Laurent polynomial ring over $k[H]$ and so $k[HX]$ is induced from $k[H]$. 

\noindent $\textnormal{(iv)} \Rightarrow \textnormal{(iii)}$

Observe that, if (iv) holds, then 
$$\dim \Delta(G)=\dim\Delta(H)+\rank(G/H)$$
Let $G_1$ be a finitely generated  dense subgroup of $G$ and $H_1=H\cap G_1$ so that $H_1$ is a finitely generated dense subgroup of $H$. 

Then $\dim\Delta(G)=\dim\Delta(G_1)$ is the transcendence degree of $k[G_1]$ over $k$ (by Theorem A of \cite{berg}) and this, in turn, is the transcendence degree of $k[G]$ over $k$. Similarly, $\dim\Delta(H)$ is the transcendence degree of $k[H]$ over $k$. Thus the rank of $G/H$ is the transcendence degree of $k[G]$ over $k[H]$.

\noindent $\textnormal{(ii)} \Rightarrow \textnormal{(iv)}$ 

Choose $X$ to be a finitely generated subgroup of $G$ so that $H\cap X=\{1\}$ and $HX$ is dense. 
Then $k[HX]$ is isomorphic to the Laurent polynomial ring over $k[H]$. It is then easily observed that a valuation on $k[H]$ can be extended to a valuation on $k[HX]$ with arbitrarily chosen values on a basis of $X$. Otherwise put, if $\chi \in (HX)^*$ and $\chi_H\in \Delta(H)$ then $\chi \in \Delta(HX)$. We can replace $HX$ by $G$ in this last sentence since $HX$ is dense in $G$ and so $(HX)^*=G^*$ and $\Delta(HX)=\Delta(G)$. Thus the complete inverse image of $\Delta(H)$ under the restriction map is all of $\Delta(G)$.
\end{proof}


\newpage

\begin{thebibliography}{99} 

\bibitem{berg} R. Bieri and J. R. J. Groves `The geometry of the set of characters induced by valuations. ',  \emph{J. Reine Angew. Math.} 347 (1984), 168-195.

\bibitem{rig} R. Bieri and J. R. J. Groves `A rigidity property for the set of all characters induced by valuations. ',  \emph{Trans. Amer. Math. Soc. } 294 (1986), 425-434.

\bibitem{cjbb} C. J. B. Brookes `Ideals in group rings of soluble groups of finite rank', \emph{Math. Proc. Cambridge Philos. Soc.} 97 (1985), 27-49.

\bibitem{seg1} D. Segal `On the group rings of abelian minimax groups. ',  \emph{J. Algebra} 237(2001), 64-94.

\bibitem{seg2} D. Segal `On the group rings of abelian minimax groups. II. The singular case. ',  \emph{J. Algebra} 306 (2006), 378-396.

\bibitem{zs} Oscar Zariski and Pierre Samuel, \emph{ Commutative Algebra; Volume 1}, Springer-Verlag, Berlin/Heidelberg/New York, 1958.

\end{thebibliography}
\end{document}